\documentclass{amsart}

\usepackage{amsmath, amsthm, amscd, amsfonts, amssymb, enumerate, verbatim, newlfont, calc, graphicx, color}
\usepackage[bookmarksnumbered, colorlinks, plainpages]{hyperref}
\usepackage{tikz} 

\newtheorem{theorem}{Theorem}[section]
\newtheorem{question}[theorem]{Question}
\newtheorem{lemma}[theorem]{Lemma}
\newtheorem{proposition}[theorem]{Proposition}
\newtheorem{corollary}[theorem]{Corollary}
\theoremstyle{definition}
\newtheorem{definition}[theorem]{Definition}

\newtheorem{remark}[theorem]{Remark}

\usepackage{pstricks}
\usepackage{epsfig}
\usepackage{pst-grad}
\usepackage{pst-plot}

\begin{document}
\author[M. Nasernejad, A. A. Qureshi, S. Bandari, and A. 
 Musapa\c{s}ao\u{g}lu]{Mehrdad ~ Nasernejad$^{1}$,  Ayesha Asloob Qureshi$^{2,*}$, Somayeh Bandari$^{3}$, 
Asli Musapa\c{s}ao\u{g}lu$^{2}$ }
\title[Dominating ideals and closed neighborhood ideals ]{Dominating ideals and closed neighborhood ideals of graphs}
\subjclass[2010]{13B25, 13F20, 05C25, 05E40.} 
\keywords {Closed neighborhood ideals, Dominating ideals, Trees, Cycles, normality, Normally torsion-freeness.}
\thanks{$^*$Corresponding author}

\thanks{Mehrdad Nasernejad  was in part supported by a grant from IPM (No. 14001300118). Ayesha Asloob Qureshi  and Asli Musapa\c{s}ao\u{g}lu were supported by The Scientific and Technological Research Council of Turkey - T\"UBITAK (Grant No: 118F169)}

\thanks{E-mail addresses:  m\_{nasernejad@yahoo.com},  aqureshi@sabanciuniv.edu, s.bandari@bzte.ac.ir, atmusapasaoglu@sabanciuniv.edu}  
\maketitle
 
\begin{center}
{\it
$^1$School of Mathematics, Institute for Research  in  Fundamental \\
Sciences (IPM),  P.O. Box 19395-5746, Tehran, Iran\\
 $^2$Sabanci University, Faculty of Engineering and Natural Sciences, \\
Orta Mahalle, Tuzla 34956, Istanbul, Turkey\\
$^3$ Department of Mathematics, Buein Zahra Technical University,\\ Buein Zahra, Qazvin, Iran
}
\end{center}

\vspace{0.4cm}

\begin{abstract}
We study the closed neighborhood ideals and the dominating ideals of graphs, in particular, some classes of trees and cycles. We prove that the closed neighborhood ideals and the dominating ideals of some classes of  trees are normally torsion-free. The closed neighborhood ideals and the dominating ideals of cycles fail to be normally torsion-free. However, we prove that the closed neighborhood ideals of cycles admit the (strong)  persistence property and the dominating ideals of cycles are nearly normally torsion-free.
\end{abstract}
\vspace{0.4cm}

\section*{Introduction}

Square-free monomial ideals have been connected to several combinatorial structures to facilitate the study of their algebraic and homological properties. Common examples of these combinatorial structures include  graphs, hypergraphs, simplicial complexes and matroids. In particular, every square-free monomial ideal generated in degree two can be viewed as an edge ideal of a simple graph. Edge ideals were introduced by Villarreal in \cite{V2} and since their first appearance, they have been a central topic of many articles. One of the interesting properties  of the edge ideals is that their minimal primes correspond to the minimal vertex covers of their underlying graphs. In other words, the Alexander dual of the edge ideal of a graph $G$ is the cover ideal of $G$, that is, a square-free monomial ideal whose minimal generators correspond  to the minimal vertex covers of the underlying graph. Inspired by this relation, the closed neighborhood ideals and the dominating ideals of graphs were recently introduced in \cite{SM}. Let $G$ be a simple graph. The closed neighborhood ideal $NI(G)$ of $G$  is generated by square-free monomials that correspond to the closed neighborhoods of the vertices of $G$, whereas, the dominating ideal $DI(G)$ of $G$ is generated by the monomials  that correspond to the dominating sets of $G$  (see Section~\ref{prem} for the formal definitions). As shown in \cite{SM}, $NI(G)$ and $DI(G)$ are the Alexander dual of each other, a similar relation that exists between edge ideals and   cover ideals of $G$. 

Domination in graphs was mathematically formulated by Berge and Ore in 1960's and  has been widely studied by many researchers due to its enormous and growing applications in various fields including computer sciences, operations research, linear algebra and optimization. Let $G$ be a simple graph with vertex set $V(G)$. A set $S\subseteq V(G)$ is known as dominating set of $G$ if every vertex in $V(G)\setminus S$ is adjacent to at least one vertex in $S$.  We refer to \cite{HHS} for further concepts related to the domination in graphs. In this paper, our main goal is to further extend the study on closed neighborhood ideals and dominating ideals. In particular, we focus on normally torsion-freeness and certain stability property of these ideals. 

A breakdown of the contents of this paper is as follows: in Section~\ref{prem}, we provide all the needed definitions and notions. In Section~\ref{trees}, we focus on the closed neighborhood ideals and the dominating ideals of star graphs. Any square-free monomial ideal can be visualized as an edge ideal of a hypergraph. A hypergraph $\mathcal{H}$ is called {\em Mengerian} if it satisfies a certain  min-max equation, which is known as the Mengerian property in hypergraph theory or as the max-flow min-cut property in integer programming. Algebraically, it is equivalent to $I(\mathcal{H})$ being normally torsion-free, see \cite[Corollary 10.3.15]{HH1}, \cite[Theorem 14.3.6]{V1}. This fact enhances the importance of normally torsion-free ideals. In Section~\ref{trees}, our main goal is to establish the normally torsion-freeness of the closed neighborhood ideals and the dominating ideals of star graphs, which is achieved in Corollaries ~\ref{star graph1} and \ref{star graph2}. To do this, we first prove some results of general nature, which provide certain inductive and recursive techniques to create new normally torsion-free ideals based on the existing ones, see Theorem~\ref{NTF2} and Lemma~\ref{Whisker1}. We apply these techniques to study the normally torsion-freeness of the closed neighborhood ideals and the dominating ideals of cone graph and whisker graph of a given graph, see Corollary~\ref{suspension} and Lemma~\ref{cone-NI}.  In addition, in  Corollary~\ref{Path.NTF}, we prove that 3-path ideals of path graphs are normally torsion-free. 

In Section~\ref{sec-cycles}, we turn our attention to the closed neighborhood ideals and the dominating ideals of cycles. The edge ideals and the cover ideals of cycles are well-studied in the context of normally torsion-freeness. It is a well-known fact that the edge ideals and cover ideals of even cycles are normally torsion-free, and odd cycles fails to have this property in general. Given a cycle $C_n$ of length $n$, it is natural to expect somewhat similar behaviour for $NI(C_n)$ and $DI(C_n)$, but we observe in Section~\ref{sec-cycles} that this is not the case. Normally torsion-freeness is not maintained by $NI(C_n)$, but we prove in Theorem~\ref{cycle1} that they admit strong persistence property, and therefore, the persistence property. This facilitates to study the behaviour of depth of powers of $NI(C_n)$ in Corollary~\ref{depth}. As a final result, we prove in Theorem~\ref{cycle2} that $DI(C_n)$ are nearly normally torsion-free.

\section{Preliminaries}\label{prem}

Let $G$ be a finite simple undirected graph with vertex set $V(G)$ and edge set $E(G)$. The degree of a vertex $v \in V(G)$ is denoted by $\deg(v)$ and it represents the number of vertices adjacent to $v$. We briefly recall some well-known notions from graph theory. A set $T\subseteq V(G)$ is called a {\em vertex cover} of $G$ if it intersects every edge of $G$ non-trivially. A vertex cover is called minimal if it does not properly contain any other vertex cover of $G$. For each vertex $v \in V(G)$, the {\em closed neighborhood} of $v$ in $G$ is defined as follows: 

\[
N_G[v]=\{u \in V(G) : \{u,v\} \in E(G)\} \cup \{v\}.
\]
When there is no confusion about the underlying graph, we will denote $N_G[v]$ simply by $N[v]$. A subset $S\subseteq V(G)$ is called {\em dominating set} of $G$ if $S \cap N[v] \neq \emptyset$, for all $v\in V(G)$. A dominating set is called minimal if it does not properly contain any other dominating set of $G$. A minimum dominating set of $G$ is a minimal dominating set with the smallest size. The {\em dominating number} of $G$, denoted by $\gamma(G)$ is the size of its minimum dominating set, that is, 
\[
\gamma(G)=\mathrm{min} \{|S|: S \text{ is a minimal dominating set of } G\}
\]

The dominating sets and domination numbers of graphs are
 well-studied topics in graph theory. We refer to \cite{HHS} for further information. 

Let $R=K[x_1, \ldots, x_n]$ be a polynomial ring over a field $K$. An ideal $I\subset R$ is called a monomial ideal if it admits a generating set consisting of monomials. Furthermore, $I$ is said to be square-free monomial ideal if it is generated by square-free monomials. Throughout the following text, the unique minimal generating set of a monomial ideal $I$ will be denoted by $\mathcal{G}(I)$. The {\em support} of a monomial $u$, denoted by $\mathrm{supp}(u)$, is the set of variables that divide $u$. Moreover, for a monomial ideal $I$, we set $\mathrm{supp}(I)=\bigcup_{u \in \mathcal{G}(I)}\mathrm{supp}(u)$.

Let $G$ be a simple graph with  $V(G)=\{1,2,\ldots, n\}$. Square-free monomial ideals of $R$ can be associated with graphs in many different ways. We recall some commonly known definitions in this context. The {\em edge ideal} of $G$, denoted by $I(G)$, is 
\[
I(G)=(x_ix_j : \{i,j\} \in E(G)).
\]
The {\em cover ideal} of $G$, denoted by $J(G)$, is 
\[
J(G)=(\prod_{i \in T}x_i: T \text{ is a minimal vertex cover of } G).
\]

Let $t$ be a fixed positive integer.  The $t$-path ideal of $G$, denoted by $I_t(G)$, is defined as  
\[
I_t(G)=(x_{i_1}x_{i_2}\cdots x_{i_t}: \{i_1, \ldots, i_t\} \text{ is a path of length $t-1$ in } G).
\]

The notion of path ideals is a generalization of edge ideals. Indeed, we have $I(G) =I_2(G)$.

In \cite{SM}, the {\em closed neighborhood ideal} of $G$ has been introduced as 
 \[
NI(G)= (\prod_{j \in N[i]} x_j:  i \in V(G)).
\]

Moreover, in \cite{SM}, the dominating ideal of $G$ is defined as
\[
DI(G)=(\prod_{i \in S}x_i: S \text{ is a minimal dominating set of } G).
\]

 Given a square-free monomial ideal $I\subset R$,  the Alexander dual of  $I$, denoted by $I^\vee$,  is given by 
 $$I^\vee= \bigcap_{u\in \mathcal{G}(I)} (x_i~:~ x_i \in \mathrm{supp}(u)).$$

It is a well-recognized fact that $J(G)$ is the Alexander dual of $I(G)$, for example, see \cite[Lemma 9.1.4]{HH1}. It is shown in \cite[Lemma 2.2]{SM} that $DI(G)$ is the Alexander dual of $NI(G)$. 

Next, we recall some notions related to hypergraphs. A finite {\it hypergraph} $\mathcal{H}$ on a vertex set $[n]=\{1,2,\ldots,n\}$ is a collection of edges  $\{ E_1, \ldots, E_m\}$ with $E_i \subseteq [n]$, for all $i=1, \ldots,m$. The vertex set $[n]$ of $\mathcal{H}$ is denoted by $V({\mathcal{H}})$, and the edge set of $\mathcal{H}$ is denoted by $E({\mathcal{H}})$. The {\it edge ideal} of $\mathcal{H}$ is given by
$$I(\mathcal{H}) = (\prod_{j\in E_i} x_j : E_i\in  E({\mathcal{H}})).$$

A subset  $W \subseteq V_{\mathcal{H}}$  is a {\it vertex cover} of $\mathcal{H}$ if $W \cap E_i\neq \emptyset$  for all $i=1, \ldots, m$. A vertex cover $W$ is {\it minimal} if no proper subset of $W$ is a vertex cover of  $\mathcal{H}$. The cover ideal of the hypergraph  $\mathcal{H}$, 
denoted by $J(\mathcal{H})$, is given by 
\[
J(\mathcal{H})=(\prod_{i \in W}x_i: W \text{ is a minimal vertex cover of } \mathcal{H}).
\]

Similar to the case of edge ideal of graphs, the cover ideal $J(\mathcal{H})$ is  the Alexander dual of $I(\mathcal{H})$, that is,  $J(\mathcal{H})=I(\mathcal{H})^{\vee}$, for example, see \cite[Theorem 6.3.39]{V1}.

Next, we recall some definitions and notions from commutative algebra. Let  $R$ be  a commutative Noetherian ring and $I$ be   an ideal of $R$. A prime ideal $\mathfrak{p}\subset  R$ is an {\it associated prime} of $I$ if there exists an element $v$ in $R$ such that $\mathfrak{p}=(I:_R v)$, where $(I:_R v)=\{r\in R |~  rv\in I\}$. The  {\it set of associated primes} of $I$, denoted by  $\mathrm{Ass}_R(R/I)$, is the set of all prime ideals associated to  $I$. 
The minimal members of $\mathrm{Ass}_R(R/I)$ are called the {\it minimal} primes of $I$, and $\mathrm{Min}(I)$ denotes the set of minimal prime ideals of $I$. Moreover, the associated primes of $I$ which are not minimal are called the {\it embedded} primes of $I$. If $I$ is a square-free monomial ideal, then $\mathrm{Ass}_R(R/I)=\mathrm{Min}(I)$, for example see \cite[Corollary 1.3.6]{HH1}.   Let $I$  be an ideal  of $R$  and  $\mathfrak{p}_1, \ldots, \mathfrak{p}_r$  be   the minimal primes of $I$. When there is no confusion about the underlying ring, we will denote the set of associated primes of $I$ simply by $\mathrm{Ass}(R/I)$. Given an integer $n \geq 1$, the {\it $n$-th symbolic
power} of $I$  is defined to be the ideal 
$$I^{(n)} = \mathfrak{q}_1  \cap \cdots \cap  \mathfrak{q}_r,$$
where $\mathfrak{q}_i$  is the primary component of $I^n$  corresponding to $\mathfrak{p}_i$. 

\begin{definition}
An ideal $I$ is called {\it normally torsion-free} if  $\mathrm{Ass}(R/I^k)\subseteq \mathrm{Ass}(R/I)$, for all $k\geq 1$. If $I$ is a square-free monomial ideal, then $I$ is normally torsion-free if and only if $I^k=I^{(k)}$, for all $k \geq 1$, see \cite[Theorem 1.4.6]{HH1}. The concept of normally torsion-free ideals is generalized in \cite{C} as follows: a monomial ideal $I$ in a polynomial  ring $R=K[x_1, \ldots, x_n]$ over a field $K$ is called {\it nearly normally torsion-free}  if there exist a positive integer $k$ and a monomial prime ideal
 $\mathfrak{p}$ such that $\mathrm{Ass}(R/I^m)=\mathrm{Min}(I)$ for all $1\leq m\leq k$, and 
 $\mathrm{Ass}(R/I^m) \subseteq \mathrm{Min}(I) \cup \{\mathfrak{p}\}$ for all $m \geq k+1$. In \cite{NQKR}, several classes of nearly normally torsion-free ideals arising from graphs and hypergraphs are discussed. By Gitler et al. \cite{GRV}, it is well-known that the cover ideals of bipartite
graphs are normally torsion-free. Furthermore, normally torsion-free square-free
monomial ideals have been studied in Ha and Morey \cite{HM},  and Sullivant \cite{S}. On the other hand, there is little known for the normally torsion-free monomial ideals which are not square-free.

 \end{definition}
 
\begin{definition}\label{spersistence}
The ideal $I \subset R$ is said to have the {\it persistence property} if $\mathrm{Ass}(R/I^k)\subseteq \mathrm{Ass}(R/I^{k+1})$ 
for all positive integers $k$. Moreover, an ideal $I$ satisfies the {\it strong persistence property} if $(I^{k+1}: I)=I^k$ for all positive integers $k$,  for more details refer to \cite{ HQ, N2}. The strong persistence property implies the persistence property, however the converse is not true, as noted in \cite{HQ}. Furthermore,  we say that $I$  has the {\it symbolic strong persistence property}  if $(I^{(k+1)}: I^{(1)})=I^{(k)}$ for all $k$, where $I^{(k)}$ denotes the  $k$-th symbolic  power  of $I$. \end{definition}

 Let $R$ be a unitary commutative ring and $I$ an ideal in $R$. An element $f\in R$ is {\it integral} over $I$, if there exists an equation 
 $$f^k+c_1f^{k-1}+\cdots +c_{k-1}f+c_k=0 ~~\mathrm{with} ~~ c_i\in I^i.$$
 The set of elements $\overline{I}$ in $R$ which are integral over $I$ is the 
 {\it integral closure} of $I$. The ideal $I$ is {\it integrally closed}, if $I=\overline{I}$, and $I$ is {\it normal} if all powers of $I$ are integrally closed, refer to \cite{HH1} for more information. The notion of integrality for a monomial ideal $I$ can be described in a simpler way as following: a monomial $u \in  R=K[x_1, \ldots, x_n]$ is integral over $I\subset R$ if and only if there exists an integer $k$ such that  $u^k \in I^k$, see \cite[Theorem 1.4.2]{HH1}.

\section{On the closed neighborhood ideals and dominating ideals of some classes of  trees}\label{trees}
In this section, our main goal is to establish that the closed neighborhood ideals and the dominating ideals of star graphs  are normally torsion-free. To do this, we will first prove several results of general nature. The next proposition is a well-known result,  but we re-prove it by a new proof.


\begin{proposition}\label{SPPtoPP}
Let $I$ be an ideal in a commutative Noetherian ring $R$ such that satisfies the strong persistence property. Then $I$ has the persistence property.
\end{proposition}
\begin{proof}
Fix $k\geq 1$, and choose an arbitrary element $\mathfrak{p} \in \mathrm{Ass}_S(S/I^k)$. This implies that $\mathfrak{p}=(I^k:_Sh)$ for some $h\in S$. Since $I$ satisfies the strong persistence property, we have $(I^{k+1}:_SI)=I^k$, and so $\mathfrak{p}=((I^{k+1}:_SI):_Sh)$. Let $\mathcal{G}(I)=\{u_1, \ldots, u_m\}$. Hence, one obtains 
$\mathfrak{p}=(I^{k+1}:_S
h\sum_{i=1}^mu_iS)=\cap_{i=1}^m(I^{k+1}:_Shu_i)$. Accordingly, we get $\mathfrak{p}=(I^{k+1}:_Shu_i)$ for some $1\leq i \leq m$. Therefore, $\mathfrak{p} \in \mathrm{Ass}_S(S/I^{k+1})$. This means that $I$ has the persistence property, as claimed. 
\end{proof}

To prove Theorem \ref{NTF2}, we need the following result. We state it here for ease of reference.

 \begin{theorem}\label{Embedded}\cite[Theorem 3.7]{SNQ}
Let $I$ be a square-free  monomial ideal in a polynomial ring $R=K[x_1, \ldots, x_n]$ over a field $K$ and $\mathfrak{m}=(x_1, \ldots, x_n)$. If there exists a square-free monomial  $v \in I$ such that $v\in \mathfrak{p}\setminus \mathfrak{p}^2$ for any $\mathfrak{p}\in \mathrm{Min}(I)$, and $\mathfrak{m}\setminus x_i \notin \mathrm{Ass}(R/(I\setminus x_i)^s)$ for all $s$ and $x_i \in \mathrm{supp}(v)$,  then the following statements hold:
\begin{itemize}
\item[(i)] $I$  is normally torsion-free. 
\item[(ii)]  $I$ is normal.
\item[(iii)]  $I$ has the strong persistence property. 
\item[(iv)]  $I$ has the persistence property. 
\item[(v)]  $I$ has the symbolic strong persistence property. 
\end{itemize}
\end{theorem}


The next theorem will be used frequently to formulate proofs of some main results of this paper. It provides a way to create new normally torsion-free ideals based on the existing ones. 

\begin{theorem} \label{NTF2}
Let $I$ be a normally torsion-free square-free monomial ideal in a polynomial ring  $R = K[x_1, \ldots, x_n]$ and $h$ be a square-free monomial in $R$. Let there exist two  variables $x_r$ and $x_s$  with $1\leq  r\neq s  \leq n$ such  
that  $\gcd(h, u)=1$ or $\gcd(h, u)=x_r$  or  $\gcd(h, u)=x_rx_s$  for all  $u \in  \mathcal{G}(I)$.  Then the following statements hold:
\begin{itemize}
\item[(i)] $I+hR$  is normally torsion-free. 
\item[(ii)] $I+hR$  is nearly normally torsion-free. 
\item[(iii)]  $I+hR$ is normal.
\item[(iv)]  $I+hR$ has the strong persistence property. 
\item[(v)]  $I+hR$ has the persistence property. 
\item[(vi)]  $I+hR$ has the symbolic strong persistence property.
\end{itemize}
\end{theorem}

\begin{proof}
(i)  For convenience of notation, put $L:=I+hR$. 
If  $L\setminus x_k=\mathfrak{m}\setminus x_k$ for some $1\leq k \leq n$, then  one can write  $L=x_kJ+\mathfrak{m}\setminus x_k$. If $J=R$, then $L=\mathfrak{m}$, and there is nothing to prove. Let  $J\neq R$, and take  an arbitrary   element $v\in \mathcal{G}(J)$. If $x_\ell \mid v$ for some $\ell \in \{1, \ldots, n\}\setminus \{k\}$, then $v\in \mathfrak{m}\setminus x_k$, and so $J\subseteq \mathfrak{m}\setminus x_k$. This implies that $L=\mathfrak{m}\setminus x_k$, and hence the assertion holds. 
We thus  assume that   $L\setminus x_k \neq \mathfrak{m}\setminus x_k$ for all  $k=1, \ldots,  n$. 
 We claim that   $h\in \mathfrak{p}\setminus \mathfrak{p}^2$ for any $\mathfrak{p}\in \mathrm{Min}(L)$. Take an arbitrary element $\mathfrak{p}\in \mathrm{Min}(L)$. 
 Since $h\in L$ and $L\subseteq \mathfrak{p}$, one has $h \in \mathfrak{p}$.
   Suppose, on the contrary, that $h\in \mathfrak{p}^2$. Due to $h$ is square-free, this gives that $|\mathrm{supp}(h) \cap \mathrm{supp} (\mathfrak{p})| \geq 2$. We observe the following:\\
(i)  If $x_s \in \mathrm{supp}(u)$ for some $u \in \mathcal{G}(I)$, then $x_r \in \mathrm{supp}(u)$ as well. It is due to the assumption on $\gcd(h,u)$ with $u \in \mathcal{G}(I)$.  \\
(ii) At most one of $x_r$ and $x_s$ can be in $\mathrm{supp}(\mathfrak{p})$. Indeed, if both $x_r,x_s \in \mathrm{supp} (\mathfrak{p})$, then $x_r,x_s \in \mathrm{supp}(h) \cap \mathrm{supp} (\mathfrak{p})$. From (i), we see that $u \in \mathfrak{p}\setminus \{x_s\}$ for all $u \in \mathcal{G}(I)$. Also, $h \in \mathfrak{p}\setminus \{x_s\}$. Hence, $L \subset \mathfrak{p}\setminus \{x_s\}$, a contradiction to the minimality of $\mathfrak{p}$. 

In order to establish our claim, we have the following cases to discuss:
\vspace{1mm}

\textbf{Case 1.}  $x_r \in \mathfrak{p}$. Take any $x _i \in \mathrm{supp}(h) \cap \mathrm{supp} (\mathfrak{p})$ such that $x_r \neq x_i$. Then $x_s\neq x_i$ due to (ii).  From  the assumption on $\gcd(h,u)$ with $u \in \mathcal{G}(I)$ it follows that $x_i \notin \mathrm{supp}(I)$. Therefore, $I \subset \mathfrak{p}\setminus \{x_i\} $. Since $ h \in \mathfrak{p}\setminus \{x_i\} $, we conclude that $L \subset \mathfrak{p}\setminus \{x_i\} $, a contradiction to the minimality of $\mathfrak{p}$.
\vspace{1mm}

\textbf{Case 2.}  $x_s \in \mathfrak{p}$. By mimicking the same argument as in Case 1, we again obtain a contradiction to the minimality of $\mathfrak{p}$.
\vspace{1mm}

\textbf{Case 3.}  $x_r \notin \mathfrak{p}$ and $x_s \notin \mathfrak{p}$. Take any $x_i, x_j \in \mathrm{supp}(h) \cap \mathrm{supp} (\mathfrak{p})$. Then $x_i , x_j \notin \mathrm{supp}(I)$,  due to the assumption on $\gcd(h,u)$ with $u \in \mathcal{G}(I)$. It yields that $I \subset \mathfrak{p}\setminus \{x_i\} $. Since $h \in \mathfrak{p}\setminus \{x_i\} $, we conclude that $L \subset \mathfrak{p}\setminus \{x_i\} $, again a contradiction to the minimality of $\mathfrak{p}$.

This shows that our claim holds true. 
To complete the proof,  note that for all  $x_k \in \mathrm{supp}(h)$,  one has    $L\setminus x_k=I\setminus x_k$.  Based on 
\cite[Theorem 3.21]{SN}, we gain $I\setminus x_k$ is normally torsion-free as well. This leads to  $L\setminus x_k$ is normally torsion-free.       Fix $s\geq 1$.   Suppose, on the contrary, that 
$\mathfrak{m}\setminus x_k \in \mathrm{Ass}(R/(L\setminus x_k)^s)$ for some $ k$. Because 
$\mathrm{Ass}(R/(L\setminus x_k)^s)=\mathrm{Min}(L\setminus x_k)$, we get  $\mathfrak{m}\setminus x_k \in\mathrm{Min}(L\setminus x_k)$, and so $L\setminus x_k=\mathfrak{m}\setminus x_k,$ which is a contradiction. Therefore, $\mathfrak{m}\setminus x_i \notin \mathrm{Ass}(R/(I\setminus x_i)^s)$ for all $s$ and $x_i \in \mathrm{supp}(h)$.   Consequently, the assertion can be concluded readily from Theorem \ref{Embedded}. \par   
 (ii) It is well-known, by \cite{NQKR},  that normally torsion-freeness implies 
 nearly normally torsion-freeness. \par  
 (iii) In view of   \cite[Theorem 1.4.6]{HH1}, every normally torsion-free square-free monomial ideal is normal. Hence, the  claim  can be deduced from (i).  \par 
(iv) According to   \cite[Theorem 6.2]{RNA}, every normal monomial ideal has the strong persistence property. Thus, the assertion  follows readily from (iii). \par 
(v)  By  Proposition \ref{SPPtoPP},  the strong persistence property implies the persistence property. Therefore, we 
  can  conclude  the claim  from (iv). \par 
(vi)  According to   \cite[Theorem 5.1]{NKRT}, every square-free monomial ideal has the  symbolic strong persistence property, and  so  the assertion  holds. 
\end{proof}


As an  immediate consequence of Theorem \ref{NTF2}, we give the following corollary. 

\begin{corollary} \label{Path.NTF}
The path ideals corresponding to path graphs of length two  are normally torsion-free.
\end{corollary}

\begin{proof}
Let $P=(V(P), E(P))$ denote a   path graph with the vertex set 
$V(P)=\{x_1, \ldots, x_n\}$ and the edge set 
$E(P)=\{\{x_i, x_{i+1}\}~:~ i=1, \ldots, n-1\} \cup \{\{x_n, x_1\}\}.$ Hence, the path ideal corresponding to the path graph $P$  of length two  is given by 
$$L:=(x_ix_{i+1}x_{i+2}~:~i=1, \ldots, n-2).$$ We proceed by induction on $n$. If $n=3$, then $L=(x_1x_2x_3)$, and there is nothing to prove. Let $n>3$ and the claim has been proven for $n-1$. Set $h:=x_{n-2}x_{n-1}x_n$ and 
$I:=(x_ix_{i+1}x_{i+2}~:~i=1, \ldots, n-3).$ One can easily check that, for each $u\in \mathcal{G}(I)$, we have $\gcd(h,u)=1$ or $\gcd(h,u)=x_{n-2}$ or $\gcd(h,u)=x_{n-2}x_{n-1}.$ It  follows from the induction hypothesis that $I$ is normally torsion-free. Since     $L=I+hR$, where $R=K[x_1, \ldots, x_n]$, we can derive the assertion from  Theorem \ref{NTF2}. 
 \end{proof}


As an application of Theorem \ref{NTF2}, we give the following lemma. 

\begin{lemma}\label{Whisker1}
Let $G=(V(G), E(G))$ and  $H=(V(H), E(H))$ be two finite simple graphs such that $V(H)=V(G)\cup \{w\}$ with $w\notin V(G)$,  and $E(H)=E(G) \cup \{\{v,w\}\}$ for some vertex $v\in V(G)$. If $NI(G)$ is normally torsion-free and $\prod_{j\in N_{G}[v]}x_j \notin \mathcal{G}(NI(G))$, then $NI(H)$ is normally torsion-free.
\end{lemma}

\begin{proof}
Let $NI(G)$ be  normally torsion-free. It is routine to check that 
$NI(H)= NI(G)+(x_vx_w)R$, where $R=K[x_{\alpha}~:~ \alpha \in V(H)]$. In addition, one can easily see that either  
$\mathrm{gcd}(x_vx_w, u)=1$  or $\mathrm{gcd}(x_vx_w, u)=x_v$ for all $u\in \mathcal{G}(NI(G))$. Therefore, the claim follows immediately from Theorem \ref{NTF2}.
\end{proof}

 
 We are ready to state  the first  main result  of this paper as an immediate corollary of Theorem~\ref{NTF2} and Lemma~\ref{Whisker1}.
 
\begin{corollary}\label{star graph1}
The closed neighborhood ideals of star graphs  are normally torsion-free.
\end{corollary}
\begin{proof}
Proceed by induction on the number of vertices  and use Lemma \ref{Whisker1}. 
\end{proof}

 
In what follows, we investigate the closed  neighborhood ideals related to  the whisker graph and cone of a graph.

\begin{definition}\cite[Definition 7.3.10]{V1}
 Let $G_0$ be a graph on the vertex set 
 $Y = \{y_1, \ldots,  y_n\}$ 
and take a new set of variables 
$X = \{x_1, \ldots , x_n\}$. The {\it whisker graph} or
{\it suspension} of $G_0$, denoted by $G_0 \cup  W(Y )$, is the graph obtained from $G_0$ by attaching to each vertex $y_i$ a new vertex $x_i$ and the edge $\{x_i, y_i\}$. The
edge $\{x_i, y_i\}$ is called a {\it whisker}.
\end{definition}
 
 \begin{question} \label{suspension}
 (i) Can we conclude that the closed neighborhood ideals of trees are normally torsion-free?
 
 (ii)  Let $G_0$ be a graph and let $H:=G_0 \cup  W(Y )$  be its
whisker graph. If $NI(G_0)$ is normally torsion-free, then can we deduce that  $NI(H)$ is normally torsion-free?
 \end{question}
 
 
 \begin{definition} \cite[Definition 10.5.4]{V1}
  The {\it cone} $C(G)$, over the graph $G$, is obtained by adding a new vertex $t$ to $G$ and joining every vertex of $G$ to $t$. 
  \end{definition} 
 
 \begin{lemma} \label{cone-NI}
  Let $G$ be a graph and let $H:=C(G)$  be its
cone. Then  the following statements hold:
\begin{itemize}
\item[(i)] $NI(G)$  is normally torsion-free if and only if $NI(H)$  is normally torsion-free. 
\item[(ii)] $NI(G)$  is nearly normally torsion-free if and only if $NI(H)$  is nearly normally torsion-free. 
\item[(iii)]  $NI(G)$  is normal if and only if $NI(H)$  is normal.
\item[(iv)]  $NI(G)$ has the  strong persistence property if and only if  $NI(H)$  has the strong persistence property. 
\item[(v)]  $NI(G)$ has the  persistence property if and only if  $NI(H)$  has the persistence property.  
\item[(vi)]  Both $NI(G)$   and $NI(H)$ have  the symbolic strong persistence property. 
\end{itemize}
 \end{lemma}
 
 \begin{proof}
 Assume that the   cone $H=C(G)$  is obtained by adding the new vertex $w$ to $G$ and joining every vertex of $G$ to $w$.
 Then one can easily see that 
 $$NI(H)=x_wNI(G)+(x_w\prod_{i\in V(G)}x_i).$$
  Since $\prod_{i\in V(G)}x_i \in NI(G)$, this implies that 
  $NI(H)=x_wNI(G)$. 
  
 (i)  This claim can be deduced from 
 \cite[Lemma 3.12]{SN}. \par 
 (ii) On account of \cite[Lemma 3.6]{NQKR}, one can derive this claim. \par  
 (iii) We can conclude this assertion by virtue of  \cite[Remark 1.2]{ANR}. \par 
 (iv) This claim is an immediate consequence of \cite[Lemma 4.5]{RNA}. \par 
 (v) Due to \cite[Theorem 5.2]{KHN2}, we can deduce this assertion.\par 
(vi) This claim follows readily from  \cite[Theorem 5.1]{NKRT}. 
   \end{proof}

   
  We recall the following definition which will be used in the proof of Lemma~\ref{cone-DI}.
 \begin{definition}  \cite[Definition 2.1]{N2}
Let $I\subset R=K[x_1, \ldots, x_n]$ be a monomial ideal with $\mathcal{G}(I)=\{u_1, \ldots, u_m\}$. Then $I$ is said to be {unisplit}, if there exists $u_i \in \mathcal{G}(I)$ such that $\gcd(u_i,u_j)=1$ for all $u_j \in \mathcal{G}(I)$ with $i \neq j$.
   
   \end{definition}
 \begin{lemma} \label{cone-DI}
  Let $G$ be a graph and let $H:=C(G)$  be its
cone. Then  the following statements hold:
\begin{itemize}
\item[(i)] $DI(G)$  is normally torsion-free if and only if $DI(H)$  is normally torsion-free. 
\item[(ii)] $DI(G)$  is nearly normally torsion-free if and only if $DI(H)$  is nearly normally torsion-free. 
\item[(iii)]  $DI(G)$  is normal if and only if $DI(H)$  is normal.
\item[(iv)]   $DI(H)$  has the strong persistence property. 
\item[(v)]   $DI(H)$  has the persistence property.  
\item[(vi)]  Both $DI(G)$ and   $DI(H)$  have  the symbolic strong persistence property.
\end{itemize}
 \end{lemma}
 
 \begin{proof}
Suppose that  the   cone $H=C(G)$  is obtained by adding the new vertex $w$ to $G$ and joining every vertex of $G$ to $w$.
 Using  \cite[Lemma 2.2]{SM} yields that    
 $$DI(H)=DI(G)+(x_w).$$
  It follows now from \cite[Lemma 3.4]{HM} that,  for all $s$, 
 \begin{equation}
 \mathrm{Ass}(DI(H)^s)=\{(\mathfrak{p}, x_w)~:~ \mathfrak{p}\in \mathrm{Ass}(DI(G)^s)\}. \label{555}
 \end{equation}
 
 (i) Let $DI(G)$ be normally torsion-free. Then the claim can be deduced from  \cite[Theorem 2.5]{SN}. Conversely, let 
 $DI(H)$ be normally torsion-free. By using \cite[Theorem 3.21]{SN}, we obtain that $DI(G)$ is  normally torsion-free.  \par 
 (ii) Based on (\ref{555}), and by considering the fact that 
   $$\mathrm{Min}(DI(H))=\{(\mathfrak{p}, x_w)~:~ \mathfrak{p}\in \mathrm{Min}(DI(G))\},$$
   one can easily show this claim. \par 
 (iii) One  concludes this assertion by  \cite[Theorem 3.12]{AJZ}. \par 
 (iv) By \cite[Definition 2.1]{N2}, $DI(H)$   is a unisplit
monomial ideal. Hence, 
\cite[Theorems  2.10 and  3.1]{N2} imply that   $DI(H)$ has the  strong persistence property.   \par 
 (v) Proposition \ref{SPPtoPP} together with (iv) yield that 
  $DI(H)$ has   the  persistence property.  \par 
(vi) This assertion follows promptly from  \cite[Theorem 5.1]{NKRT}. 
     \end{proof}
   

 Our next goal is to show that the dominating ideals of star graphs  are normally torsion-free. To do this, we first prove some results of general nature. We recall some definitions from \cite{FHV2} which  are necessary to establish Theorem~\ref{Coloring2}.  Let  $\mathcal{H}= (V(\mathcal{H}), E({\mathcal{H}}))$ be a hypergraph with $V(\mathcal{H})=\{x_1, \ldots, x_n\}$.

\begin{definition} (see \cite[Definition 2.7]{FHV2})
A {\em $d$-coloring} of $\mathcal{H}$ is any partition of $V(\mathcal{H}) = C_1\cup \cdots \cup C_d$
into $d$ disjoint sets such that for every $E \in E({\mathcal{H}})$, we have $E\nsubseteq C_i$  for all $i = 1, \ldots ,d$. (In the case of a
graph $G$, this simply means that any two vertices connected by an edge receive different colors.) The
$C_i$'s are called the color classes of $\mathcal{H}$. Each color class $C_i$ is an {\em independent set}, meaning that $C_i$ does not
contain any edge of the hypergraph. The chromatic number of $\mathcal{H}$, denoted by $\chi(\mathcal{H})$, is the minimal $d$
such that $\mathcal{H}$  has a $d$-coloring.
\end{definition}
\begin{definition} (see \cite[Definition 2.8]{FHV2})
 The hypergraph $\mathcal{H}$ is called {\em critically $d$-chromatic} if 
$\chi(\mathcal{H})= d$, but for every vertex
$x\in V(\mathcal{H})$, $\chi(\mathcal{H}\setminus \{x\})< d$,   where 
$\mathcal{H}\setminus \{x\}$ denotes the hypergraph $\mathcal{H}$ with $x$ and all edges containing $x$ removed.
\end{definition}

\begin{definition} (see \cite[Definition 4.2]{FHV2})
For each $s$, the {\em $s$-th expansion} of $\mathcal{H}$ is defined to be the hypergraph obtained by replacing each vertex $x_i \in V(\mathcal{H})$ by a collection $\{x_{ij}~|~ j=1, \ldots, s\}$, and replacing $E(\mathcal{H})$ by the edge  set that  consists of edges 
 $\{x_{i_1l_1}, \ldots, x_{i_rl_r}\}$ whenever 
 $\{x_{i_1}, \ldots, x_{i_r}\}\in E(\mathcal{H})$ and edges 
 $\{x_{il}, x_{ik}\}$ for $l\neq k$. We denote this hypergraph by $\mathcal{H}^s$. The new variables $x_{ij}$ are called the shadows of $x_i$. The process of setting $x_{il}$ to equal to $x_i$ for all $i$ and $l$ is called the {\em depolarization}. 
 \end{definition}
 
 The following result is a slight generalized form of \cite[Theorem 4.9]{NQKR}.  
   
 \begin{theorem} \label{Coloring2}
Assume that  $\mathcal{G}=(V(\mathcal{G}), E(\mathcal{G}))$ and  $\mathcal{H}=(V(\mathcal{H}), E(\mathcal{H}))$ are 
 two finite simple hypergraphs such that $V(\mathcal{H})=V(\mathcal{G})\cup \{w_1, \ldots, w_t\}$ with $w_i\notin V(\mathcal{G})$ for each $i=1, \ldots, t$,  and $E(\mathcal{H})=E(\mathcal{G}) \cup \{\{v,w_1, \ldots, w_t\}\}$ for some vertex $v\in V(\mathcal{G})$. Then 
$$\mathrm{Ass}_{R'}(R'/J(\mathcal{H})^s)=\mathrm{Ass}_{R}(R/J(\mathcal{G})^s)\cup
\{(x_v, x_{w_1}, \ldots, x_{w_t})\},$$
 for all  $s$, where $R=K[ x_\alpha : \alpha\in V(\mathcal{G})]$ and 
$R'=K[ x_\alpha : \alpha\in V(\mathcal{H})]$.
\end{theorem}
\begin{proof}
For convenience of notation, set $I:=J(\mathcal{G})$ and $J:=J(\mathcal{H})$. We first prove  that 
$\mathrm{Ass}_{R}(R/I^s)\cup
\{(x_v, x_{w_1}, \ldots, x_{w_t})\}\subseteq \mathrm{Ass}_{R'}(R'/J^s)$ for all  $s$. Fix $s\geq 1$, and assume that  $\mathfrak{p}=(x_{i_1}, \ldots, x_{i_r})$ is an  arbitrary element of  $\mathrm{Ass}_R(R/I^s)$. According to 
\cite[Lemma 2.11]{FHV2}, we get  
$\mathfrak{p}\in \mathrm{Ass}(K[\mathfrak{p}]/J(\mathcal{G}_\mathfrak{p})^s)$, where $K[\mathfrak{p}]=K[x_{i_1}, \ldots, x_{i_r}]$ and  $\mathcal{G}_\mathfrak{p}$ is the induced subhypergraph  of $\mathcal{G}$ on the vertex set $\{i_1, \ldots, i_r\}\subseteq V(\mathcal{G})$. Since   $\mathcal{G}_\mathfrak{p}= \mathcal{H}_\mathfrak{p}$, we have  
$\mathfrak{p}\in \mathrm{Ass}(K[\mathfrak{p}]/J(H_\mathfrak{p})^s)$. This yields  that  $\mathfrak{p}\in \mathrm{Ass}_{R'}(R'/J^s)$. On account  of  $(x_v, x_{w_1}, \ldots, x_{w_t})\in \mathrm{Ass}_{R'}(R'/J^s)$, one derives $$\mathrm{Ass}_{R}(R/I^s)\cup \{(x_v, x_{w_1}, \ldots, x_{w_t})\}\subseteq \mathrm{Ass}_{R'}(R'/J^s).$$  
To complete the proof, it is enough for us  to show  the reverse inclusion. Assume that $\mathfrak{p}=(x_{i_1}, \ldots, x_{i_r})$ is
 an  arbitrary element of  $\mathrm{Ass}_{R'}(R'/J^s)$ with 
 $\{i_1, \ldots, i_r\}\subseteq V(\mathcal{H})$. If $\{i_1, \ldots, i_r\}\subseteq V(\mathcal{G})$, then \cite[Lemma 2.11]{FHV2} implies   that $\mathfrak{p}\in \mathrm{Ass}_R(R/I^s)$, and the proof is done. Thus, let 
 $\{{w_1}, \ldots, {w_t}\} \cap  \{i_1, \ldots, i_r\}\neq \emptyset$. 
 It follows from  \cite[Corollary 4.5]{FHV2} that  the associated  primes of $J(\mathcal{H})^s$ will correspond to critical chromatic subhypergraphs of size $s+1$ in the $s$-th expansion of $\mathcal{H}$. This means that one can take the induced subhypergraph on the vertex set $\{i_1, \ldots, i_r\}$, and 
 then form the $s$-th expansion on this induced subhypergraph, and within this new hypergraph find a critical $(s+1)$-chromatic hypergraph.   Notice that since this expansion cannot have any critical chromatic subgraphs, this implies that $\mathcal{H}_{\mathfrak{p}}$ must be connected.   
  Without loss of generality, one may  assume  that $i_1=v$ and $i_2=w_1, i_3=w_2, \ldots, i_{t+1}=w_t$. 
 Thanks to  $w_1, \ldots, w_t$ are   connected to $v$ in the hypergraph $\mathcal{H}$, and because  this induced subhypergraph is critical, if we remove any  vertex $w_k$ for some $1\leq k \leq t$, one  can color the resulting hypergraph with at least $s$ colors. This leads to  that $w_k$ has to be adjacent to at least $s$ vertices. But the only things $w_k$ is adjacent to are the shadows of $w_i$ for each $i=1, \ldots, t$, and the shadows of $v$, and so one has a clique among these vertices.  Accordingly, 
 $w_k$ and its neighbors will form a clique of size $s+1$. Since 
   a clique is a critical graph, it follows that we do not   need any element of $\{i_{t+2}, \ldots, i_r\}$ or their shadows when making the critical $(s+1)$-chromatic hypergraph. Hence, we obtain  $\mathfrak{p}=(x_v, x_{w_1}, \ldots, x_{w_t})$. This finishes the proof.    
\end{proof}


 \begin{lemma} \label{NTF1} 
Let $I$ be a normally torsion-free square-free  monomial ideal in a polynomial ring $R=K[x_{1},\ldots ,x_{n}]$ with $\mathcal{G}(I) \subset R$. Then the ideal $$L:=IS\cap (x_{n}, x_{n+1}, x_{n+2}, \ldots, x_m)\subset  S=R[x_{n+1}, x_{n+2}, \ldots, x_m],$$  is normally torsion-free.
\end{lemma}

\begin{proof}
It  is well-known that  one can  view the  square-free monomial ideal $I$ as the cover ideal of a 
 simple hypergraph $\mathcal{H}$ such that  the hypergraph $\mathcal{H}$ corresponds  to $I^\vee$, where $I^\vee$ denotes the  Alexander dual of $I$. Then  we have $I=J(\mathcal{H})$, where 
 $J(\mathcal{H})$ denotes the cover ideal of the hypergraph $\mathcal{H}$.  Fix $k\geq 1$. 
 On account of  Theorem \ref{Coloring2}, we get the following equality 
  $$\mathrm{Ass}_{S}(S/L^k)=\mathrm{Ass}_{R}(R/J(\mathcal{H})^k)\cup \{(x_{n}, x_{n+1}, x_{n+2}, \ldots, x_m)\}.$$
Because  $I$ is normally torsion-free, one derives  that 
$\mathrm{Ass}_{R}(R/J(\mathcal{H})^k)=\mathrm{Min}(J(\mathcal{H}))$, and hence $\mathrm{Ass}_{S}(S/L^k)=\mathrm{Min}(J(\mathcal{H}))\cup
\{(x_{n}, x_{n+1}, x_{n+2}, \ldots, x_m)\}.$ This gives rise to  $\mathrm{Ass}_{S}(S/L^k)=\mathrm{Min}(L)$. Therefore,  
$L$ is normally torsion-free, as claimed.
\end{proof}
 

\begin{lemma}\label{Whisker2}
Let $G=(V(G), E(G))$ and  $H=(V(H), E(H))$ be two finite simple graphs such that $V(H)=V(G)\cup \{w\}$ with $w\notin V(G)$,  and $E(H)=E(G) \cup \{\{v,w\}\}$ for some vertex $v\in V(G)$ and  $\prod_{j\in N_{G}[v]}x_j \notin \mathcal{G}(NI(G))$. If $DI(G)$ is normally torsion-free, then $DI(H)$ is normally torsion-free.
\end{lemma}

\begin{proof}
Let $DI(G)$ be  normally torsion-free. It follows from  \cite[Lemma 2.2]{SM}    that  
$DI(H)= DI(G)\cap (x_v, x_w)R$, where $R=K[x_{\alpha}~:~ \alpha \in V(H)]$. Now, we can conclude the assertion from   Lemma   \ref{NTF1}.
\end{proof}

 We are in a position to give   the second   main result  of this paper in the following corollary, which is related to dominating  ideals of star graphs.

\begin{corollary}\label{star graph2}
The dominating  ideals of star graphs   are normally torsion-free.
\end{corollary}
\begin{proof}
We use the  induction on the number of vertices  together with   Lemma \ref{Whisker2}. 
\end{proof}

 
 \begin{question}
 (i) Can we conclude that the dominating ideals of  trees are normally torsion-free?
 
 (ii)  Let $G_0$ be a graph and let $H:=G_0 \cup  W(Y )$  be its
whisker graph. Then if $DI(G_0)$ is normally torsion-free, then can we deduce that $DI(H)$ is normally torsion-free?
 \end{question}
 


 \section{On the closed neighborhood ideals and dominating ideals of cycles}\label{sec-cycles}

As stated in the introduction, the edge ideals and the cover ideals of bipartite graphs are known to be normally torsion-free,  see \cite{GRV, SVV}. In particular, the edge ideals and the cover ideals of even cycles are normally torsion-free. However, this behaviour changes when we consider the odd cycles. The cover ideals of odd cycles happen to be nearly normally torsion-free,  see \cite{NKA}, but edge ideals of odd cycles do not admit such tamed behaviour for the set of their associated primes. Given these facts, it is natural to expect some irregularities for the closed neighborhood ideals and dominating ideals of even and odd cycles. It can be verified by using Macaulay2~\cite{GS} that in general, the closed neighborhood ideals of cycles, regardless of the parity of their lengths,  are neither normally torsion-free nor nearly normally torsion-free. However, in this section, we will show that the closed neighborhood ideals of cycles admit strong persistence property. On the other side, as another main result of this section, we will show that the dominating ideals of cycles are nearly normally torsion-free. 

To establish above-mentioned results, we begin by proving the following theorem which gives an inductive way to study the normality of an ideal.   

\begin{theorem} \label{I+xH}
Let $I$ and $H$  be two  normal square-free monomial ideals in a  polynomial ring $R=K[x_1, \ldots, x_n]$ 
such that $I+H$ is normal.
  Let  $x_{c} \in \{x_1, \ldots, x_n\}$ be  a variable with  
 $\gcd (v,x_c)=1$ for all $v\in \mathcal{G}(I)\cup  \mathcal{G}(H)$. Then $L:=I+x_cH$ is normal. 
\end{theorem}

\begin{proof}
Let $\mathcal{G}(I)= \{u_{1},\ldots ,u_{s}\}$ and 
$\mathcal{G}(H)=\{h_1, \ldots, h_r\}$. 
Since  $\gcd (v,x_c)=1$ for all $v\in \mathcal{G}(I)\cup  \mathcal{G}(H)$, without loss of generality, one may
assume that $x_c=x_{1}\in K[x_{1}]$ and 
$$\mathcal{G}(I) \cup \mathcal{G}(H)=\{u_{1},\ldots ,u_{s}, h_1, \ldots, h_r\}\subseteq K[x_{2},\ldots ,x_{n}].$$   We must show  that $\overline{L^{t}}=L^{t}$ for all
integers $t\geq 1$. For this purpose, it is enough  to prove  that $%
\overline{L^{t}}\subseteq L^{t}$. Let $\alpha$ be a monomial in $\overline{L^{t}}$ and
write $\alpha =x_1^{b}\delta $ with $x_1\nmid \delta $ and $\delta \in R$. On account of  
\cite[Theorem 1.4.2]{HH1}, $\alpha ^{k}\in L^{tk}$ for some integer $k\geq 1$%
. Write%
\begin{equation}
\alpha^{k}=x_1^{bk}\delta^{k}=\prod\limits_{i=1}^{s}u_{i}^{p_{i}}
x_1^{q+\varepsilon}
\prod\limits_{j=1}^{r}h_{j}^{q_{j}}\beta \text{,} \label{11}
\end{equation}%
with $\sum_{i=1}^{s}p_{i}=p$, $\sum_{j=1}^{r}q_j=q$,
$p+q=tk$, $\varepsilon \geq 0$, and $\beta$ is some
monomial in $R$ such that $x_1\nmid \beta$. 
Because $x_1\nmid \beta$, $x_1\nmid \delta$,  and 
$\gcd (v,x_1)=1$ for all $v\in \mathcal{G}(I)\cup 
 \mathcal{G}(H)$, one can conclude   that 
$bk=q+\varepsilon $. Accordingly, by  virtue of (\ref{11}),  we
obtain 
$$\delta^{k}=\prod\limits_{i=1}^{s}u_{i}^{p_{i}}
\prod\limits_{j=1}^{r}h_{j}^{q_{j}}\beta 
 \in (I+H)^{tk}.$$ This leads to  $\delta \in \overline{(I+H)^{t}}.$ Thanks to  $I+H$ is
normal, we deduce that  $\overline{(I+H)^{t}}=(I+H)^{t}$, and so $\delta \in (I+H)^{t}$. Therefore,  one can 
write%
\begin{equation}
\delta =\prod\limits_{i=1}^{s}u_{i}^{l_{i}}
\prod\limits_{j=1}^{r}h_{j}^{z_{j}}
\gamma \text{,} \label{22}
\end{equation}%
with $\sum_{i=1}^{s}l_{i}=l$, $\sum_{j=1}^{r}z_{j}=z$, $l+z=t$,  and $\gamma $ is some monomial in 
$R$. Note that $x_1\nmid \gamma $ as $x_1\nmid \delta $. Due to  $x_1^{bk}\delta^{k}\in L^{tk}$, it follows immediately  from  (\ref{22}) that 
$$\prod\limits_{i=1}^{s}u_{i}^{l_{i}k}x_1^{bk}
\prod\limits_{j=1}^{r}h_{j}^{z_{j}k}
\gamma ^{k}\in L^{tk}=\left(
I+x_1H\right)^{tk}.$$ Consequently, we conclude that $bk\geq zk$, that is, $%
b\geq z$. This gives rise to $$x_1^{b}\delta
=\prod\limits_{i=1}^{s}u_{i}^{l_{i}}x_1^{b}
\prod\limits_{j=1}^{r}h_{j}^{z_{j}}\gamma \in \left(
I+x_1H\right)^{t},$$ and the proof is over.
\end{proof}


We state the third main result of this paper in the next theorem, which is related to the closed neighborhood ideals of cycles.  
 
\begin{theorem} \label{cycle1}
Let $C_n$ be a cycle graph of order $n$. 
 Then the following statements hold:
\begin{itemize}
\item[(i)]   $NI(C_n)$ is normal.
\item[(ii)]  $NI(C_n)$ has the strong persistence property. 
\item[(iii)]  $NI(C_n)$ has the persistence property. 
\end{itemize}
\end{theorem}
\begin{proof}
(i) Let $C_n=(V(C_n), E(C_n))$ be a cycle graph of order $n$ with $V(C_n)=\{x_1, \ldots, x_n\}$ and 
$E(C_n)=\{\{x_i, x_{i+1}\}~:~ i=1, \ldots, n-1\} \cup \{\{x_n, x_1\}\}.$ Then the closed neighborhood ideal of $C_n$ is given by 
$$NI(C_n)=(x_ix_{i+1}x_{i+2}~:~ i=1, \ldots, n) \subset R=K[x_1, \ldots, x_n],$$
where $x_{n+1}$ (respectively, $x_{n+2}$) represents $x_1$ (respectively, $x_2$).  If $n=3$, then 
$NI(C_3)=(x_1x_{2}x_{3})$, and so there is nothing to prove. Thus, let   $n\geq 4$. Put $H:=(x_2x_3, x_{n-1}x_n, 
x_2x_n)$ and  
$I:=(x_ix_{i+1}x_{i+2}~:~ i=2, \ldots, n-2).$ One can easily see that $NI(C_n)=I+x_1H$. Our strategy is to use Theorem~\ref{I+xH} to complete the proof. To do this, we first show that $I$, $H$, and $I+H$ are normal. Assume that  $G$ is  a path graph  with   $V(G)=\{x_2, x_3, x_{n-1},  x_n\}$ and  $E(G)=\{\{x_2, x_3\}, \{x_{n-1}, x_n\}, \{x_2,x_n\}\}$.  
It is routine to check that $I(G)=H$, where $I(G)$ denotes the  edge  ideal of $G$. Since, by \cite[Corollary 2.6]{GRV}, the edge ideal of any path graph is  normally torsion-free, and by remembering this fact that   every normally torsion-free square-free monomial ideal is normal, we deduce that $H$ is a normal square-free monomial ideal. Now, assume that $P$ is  a path graph with 
$V(P)=\{x_2, x_3,\ldots, x_{n-1}, x_n\}$ and 
$E(P)=\{\{x_i, x_{i+1}\}~:~ i=2, \ldots, n-1\}.$ 
 It is not hard to check that $I=I_3(P)$, where $I_3(P)$ denotes the path ideal of  length $2$ of $P$.  It follows readily from 
Corollary~\ref{Path.NTF} that $I=I_3(P)$ is normally torsion-free, and so is normal. To complete the proof, we show that $I+H$ is normal. To accomplish this, we note that 
$$I+H=(x_2x_3, x_{n-1}x_n, x_2x_n, x_ix_{i+1}x_{i+2}~:~ i=3, \ldots, n-3).$$ 
Set $A:=(x_3, x_n)$ and $B:=(x_{n-1}x_n,  x_ix_{i+1}x_{i+2}~:~ i=3, \ldots, n-3)$. Notice that 
$I+H=B+x_2A$. 
It is clear that $A$ is a normal ideal. Furthermore, it follows from Corollary~\ref{Path.NTF} and Theorem  \ref{NTF2} that 
$B$ is normally torsion-free, and so is normal. In addition, we have 
$$B+A=(x_3, x_n, x_ix_{i+1}x_{i+2}~:~ i=4, \ldots, n-3).$$
One can easily conclude from  Corollary~\ref{Path.NTF} and Theorem  \ref{NTF2} that $B+A$ is normally torsion-free, and hence is normal. By virtue of Theorem \ref{I+xH}, we deduce that $B+x_2A$ is normal, and so $I+H$ is normal as well. 
Finally,  note that 
 $\gcd (v,x_1)=1$ for all $v\in \mathcal{G}(I)\cup  \mathcal{G}(H)$. This finishes the proof. \par 
The claims (ii) and (iii) can be proven similar to parts (iv) and (v) in Theorem \ref{NTF2}. 
 \end{proof}


 The neighborhood ideals of cycles are particularly nice because they are generated by monomial of the  same degree. This fact together with Theorem~\ref{cycle1} enables us to study the depth of powers of $NI(C_n)$. For this purpose, we first recall the following definition and result from \cite{HQ}.

\begin{definition}\label{linearrelationgraphdef}
Let $I\subset R$ be a monomial ideal with $\mathcal{G}(I)=\{u_1, \ldots, u_m\}$. The {\em linear relation graph} $\Gamma_I$ of $I$ is the graph with the edge set 
\[
E(\Gamma_I)=\{\{x_i,x_j\}: \text{there exist $u_k$,$u_l \in \mathcal{G}(I)$ such that $x_iu_k=x_ju_l$}\},
\]
 and the vertex set $V(\Gamma_I)=\bigcup_{\{x_i,x_j\}\in E(\Gamma)}\{i,j\}$. 
 \end{definition}
 
\begin{theorem}\cite[Theorem 3.3]{HQ}\label{Thmdepth}
Let $I \subset R=K[x_1, \ldots, x_n]$ be a monomial ideal generated in a single degree whose linear relation graph has $r$ vertices and $s$ connected components. Then 
\[
\mathrm{depth}(R/I^t) \leq n-t-1 \text{ for } t=1, \ldots, r-s.
\]
\end{theorem}

In order to apply above theorem, we first analyze the linear relation graph of $NI(C_n)$. Let $V(C_n)=[n]$ and $E(C_n)=\{\{1,2\}, \{2,3\}, \ldots, \{n-1,n\}, \{n,1\}\}$. We set the following notations.

\begin{enumerate}
\item $u_i=\prod_{j \in N[i]} x_j$. In simple words, $u_i$ is the monomial that corresponds to the closed neighborhood of the vertex $i$. 

\item Note that $u_i=x_{i-1}x_ix_{i+1}$, for all $i=2, \ldots n-1$ and $u_1=x_nx_1x_2$, $u_n=x_{n-1}x_n x_{1}$. To synchronize this notation for all $i$, if $i>n$ then we read $i$ as $i (\mathrm{mod}\; n)$. In this way, we can write $u_i=x_{i-1}x_ix_{i+1}$, for all $i=1, \ldots n$.

\end{enumerate}

\begin{remark} \label{rem1}
Let $i \neq j$. Note that each variable $x_i$ appears in exactly three monomials in $\mathcal{G}(NI(C_n))$, and these monomials are $u_{i-1}= x_{i-2}x_{i-1}x_i$, $u_i=x_{i-1}x_{i}x_{i+1}$  and $u_{i+1}=x_ix_{i+1}x_{i+2}$. From this observation, we conclude that $\{x_i,x_j\} \in E(\Gamma)$ if and only if  there exists a path of length three from $i$ to $j$ in $C_n$.  
Here a path $P$ of length $n$ is defined on $n+1$ vertices and $n$ edges.
\end{remark}

\begin{remark}\label{rem2}
Let $n \geq 4$, and set $I_n:=NI(C_n)$. Remark~\ref{rem1} leads us to the following:
\begin{enumerate}

 \item $|V(\Gamma_{I_n})| = n$. This can be easily verified because for every $i$, we can find another vertex $j$ such that there is a path of length three from $i$ to $j$ in $C_n$. \\
 
 \item $\Gamma_{I_n}$ has one connected component if $n\neq 3k$, for all $k \geq 2$.  Indeed, if \\$n=1 (\mathrm{mod}\; 3)$, that is, $n=3k+1$ for some $k \geq 1$, then we have 

\[E(\Gamma_{I_n})=\{ \{x_1, x_4\}, \{ x_4,x_7\}, \ldots, \{x_{3k-2},x_{3k+1}\},
\]
\[ 
\{ x_{3k+1}, x_{3}\}, \{x_3,x_6\}, \ldots, \{x_{3k-3}, x_{3k}\},
\]
\[
\{x_{3k}, x_2\}, \{x_2, x_5\}, \ldots, \{x_{3k-2}, x_1\}\}.
\]

If $n=2 (\mathrm{mod}\; 3)$, that is, $n=3k+2$ for some $k \geq 1$, then we have 

\[E(\Gamma_{I_n})=\{ \{x_1, x_4\}, \{ x_4,x_7\}, \ldots, \{x_{3k-2},x_{3k+1}\},
\]
\[ 
\{ x_{3k+1}, x_2\}, \{x_2,x_5\}, \ldots, \{x_{3k-1}, x_{3k+2}\},
\]
\[
\{x_{3k+2}, x_3\}, \{x_3, x_6\}, \ldots, \{x_{3k}, x_1\}\}.
\]

\item $\Gamma_{I_n}$ has three connected components if $n=3k$, for some $k \geq 2$. Set
$V(\Gamma_1)=\{x_1, x_4, \ldots, x_{3k-2}\}$, and 
\[
E(\Gamma_1)=\{ \{x_1,x_4\}, \{x_4,x_7\},\ldots,  \{x_{3k-2},x_1\} \}.
\]
Set $V(\Gamma_2)=\{x_2, x_5, \ldots, x_{3k-1}\}$, and 
\[
E(\Gamma_1)=\{ \{x_2,x_5\}, \{x_5,x_8\},\ldots,  \{x_{3k-1},x_2\} \}.
\]
Set $V(\Gamma_3)=\{x_3, x_6, \ldots, x_{3k}\}$, and 
\[
E(\Gamma_1)=\{ \{x_3,x_6\}, \{x_6,x_9\},\ldots,  \{x_{3k},x_3\} \}.
\]
It can be easily verified that $\Gamma_{I_n}$  is the disjoint union of $\Gamma_1$, $\Gamma_2$, and $\Gamma_3$.
\end{enumerate}
\end{remark}

Theorem~\ref{Thmdepth} together with Remark~\ref{rem2} leads to the following corollary:

\begin{corollary}\label{depth}
Let $n \neq 0(\mathrm{mod}\; 3)$. Set $I_n=NI(C_n)\subset R=K[x_1, \ldots, x_n]$. Then $\mathrm{depth}(R/I_n^{n-1})=0$. In particular, $\mathfrak{m} \in \mathrm{Ass}(R/I_n^{n-1})$ and $\mathrm{lim}_{k \rightarrow \infty}\mathrm{depth} R/I_n^k=0$. 
\end{corollary}


We provide the fourth  main result of this paper in the subsequent  theorem, which is related to the dominating  ideals of cycles. We will use the following result to establish our proof. 

\begin{corollary} \label{Cor.1} \cite[Corollary 3.3]{NQKR}
Let $I$ be a square-free  monomial ideal in a polynomial ring $R=K[x_1, \ldots, x_n]$ over a field $K$.  Let   $I(\mathfrak{m}\setminus \{x_i\})$ be  normally torsion-free  for all $i=1, \ldots, n$. Then $I$ is nearly normally torsion-free. 
\end{corollary}

Now, we state the next main result.

\begin{theorem} \label{cycle2}
The dominating  ideals of cycles are nearly normally torsion-free. 
\end{theorem} 
\begin{proof}
Let $C_n$ denote a cycle graph of order $n$ with 
$V(C_n)=\{x_1, \ldots, x_n\}$ and 
$E(C_n)=\{\{x_i, x_{i+1}\}~:~ i=1, \ldots, n-1\} \cup \{\{x_n, x_1\}\}$. In the light of \cite[Lemma 2.2]{SM},  the dominating  ideal of $C_n$ is given by 
$$DI(C_n)=\bigcap_{i=1}^n(x_i, x_{i+1}, x_{i+2})\subset R=[x_1, \ldots, x_n],$$
where $x_{n+1}$ (respectively, $x_{n+2}$) represents $x_1$ 
(respectively, $x_2$). Set $I:=DI(C_n)$. Our strategy is to use Corollary \ref{Cor.1}. To do this, we must show that $I(\mathfrak{m}\setminus \{x_i\})$ is   normally torsion-free  for all $i=1, \ldots, n$, where $\mathfrak{m}=(x_1, \ldots, x_n)$. Without loss of generality, it is sufficient for us to prove that $I(\mathfrak{m}\setminus \{x_1\})$ is   normally torsion-free. 
To simplify notation, set $F:=\bigcap_{i=2}^{n-2}(x_{i}, x_{i+1}, x_{i+2})$. By virtue of  Corollary \ref{Cor.1}, 
one has to show that 
the ideal $F=I(\mathfrak{m}\setminus \{x_1\})$ is normally torsion-free. To do this, let $T=(V(T), E(T))$ be the rooted  tree with the root $2$, the  vertex set $V(T)=\{x_2, \ldots, x_{n}\}$, and the edge set 
$E(T)=\{(x_{i}, x_{i+1})~:~ i=2, \ldots, n-1\}$, where 
$(x_{i}, x_{i+1})$ denotes the directed edge from the vertex $x_i$ to the vertex $x_{i+1}$ for all $i=2, \ldots, n-1$. It is not hard to check that $F$ is the  Alexander dual
of the path ideal  generated by all paths of length 2 in the rooted tree $T$. Now, one can  deduce from 
\cite[Theorem 3.2]{KHN1} that $F=I(\mathfrak{m}\setminus \{x_1\})$ is normally torsion-free. This completes the proof.  
\end{proof}


  \noindent{\bf Acknowledgments.}


We   would like to thank Professor Adam Van Tuyl for his valuable comments  in preparation of Theorem \ref{Coloring2}.
Moreover, Mehrdad Nasernejad  was in part supported by a grant  from IPM (No. 14001300118). Ayesha Asloob Qureshi and Asli Musapa\c{s}ao\u{g}lu were supported by The Scientific and Technological Research Council of Turkey - T\"UBITAK (Grant No: 118F169).


\end{document}